\documentclass[10pt]{amsart}
\usepackage{amssymb}
\usepackage{graphicx}
\usepackage{amsmath}

\usepackage{hyperref}
\hypersetup{
	pdfstartview={XYZ null null 1.00}, 
	pdfpagemode=UseNone, 
	colorlinks,
	breaklinks, 
	linkcolor=blue,
	urlcolor=blue, 
	citecolor=blue
}

\newtheorem{theorem}{Theorem}[section]

\newtheorem{lemma}[theorem]{Lemma}
\newtheorem{prop}[theorem]{Proposition}
\newtheorem{cor}[theorem]{Corollary}
\newtheorem{conj}[theorem]{Conjecture}
\newtheorem{claim}[theorem]{Claim}

\theoremstyle{definition}
\newtheorem{definition}[theorem]{Definition}

\newcommand{\zz}{(\mathbb{Z}\slash2\mathbb{Z}) *(\mathbb{Z}\slash2\mathbb{Z})}
\newcommand{\e}{\mathrm{Ends}}

\title{The uniform Gardner conjecture and rounding Borel flows}
\author{Matthew Bowen}

\address{Department of Mathematics and Statistics, McGill University, 805 Sherbrooke St W., H3A 0B9  Montreal, Canada}

\email{ matthew.bowen2@mail.mcgill.ca}

\author{G\'abor Kun}

\address{Alfr\'ed R\'enyi Institute of Mathematics, H-1053 Budapest, Re\'altanoda u. 13-15., Hungary}
\address{Institute of Mathematics, E\"otv\"os L\'or\'and University, P\'azm\'any P\'eter s\'et\'any 1/c, H-1117 Budapest, Hungary}

\email{kungabor@renyi.hu}

\author{Marcin Sabok}

\address{Department of Mathematics and Statistics, McGill University, 805 Sherbrooke St W., H3A 0B9  Montreal, Canada}
\email{marcin.sabok@mcgill.ca}

\thanks{The first and the third authors are partly funded by the NSERC Discovery Grant  RGPIN-2020-05445, NSERC Discovery Accelerator
Supplement  RGPAS-2020-00097 and NCN Grant Harmonia  2018/30/M/ST1/00668. The second author's work on the project leading to this application has received funding from the European Research Council (ERC) under the European Union's Horizon 2020 research and innovation programme (grant agreement No. 741420), from the \'UNKP-20-5 New National Excellence Program of the Ministry of Innovation and Technology from the source of the National Research, Development and Innovation Fund, from the J\'anos Bolyai Scholarship of the Hungarian Academy of Sciences and from Lend\"ulet grant no. 2022-58.}

\begin{document}

\maketitle

\begin{abstract}
We study groups which satisfy Gardner's equidecomposition conjecture for uniformly distributed sets. 
We prove that an amenable group has this property 
if and only if it does not admit $\zz$ as a quotient by a finite subgroup. Our technical contribution is an algorithm for rounding Borel flows for actions of amenable groups.
\end{abstract}

\section{Introduction}


Let $\Gamma$ be a group acting on a standard probability space $X$. Two sets $A, B \subseteq X$ are $\Gamma$-{\it equidecomposable} if there are finite partitions $A=\bigcup_{i=1}^k A_{i}, B=\bigcup_{i=1}^k B_{i}$ and group elements $\gamma_1, \dots ,\gamma_k \in \Gamma$ such that $A_i=\gamma_i B_i$ holds for every $1 \leq i \leq k$. An equidecomposition is {\it measurable} if the sets in the partitions are measurable. We recommend M\'ath\'e  \cite{MatheICM} as a survey on measurable equidecompositions.

The most famous open question on equidecompositions was Tarski's circle squaring problem \cite{Tarski.problem}. Laczkovich solved this positively \cite{lacz11}, but his construction was not measurable, which left the measurable circle squaring problem open for decades.
This was only recently solved by Grabowski, M\'ath\'e and Pikhurko \cite{gmp}, who found a measurable equidecomposition \cite{gmp}, and Marks and Unger \cite{marks-unger} found a Borel equidecomposition. 
The measurable circle squaring also can be deduced from the result of the authors \cite{bowen2021perfect} on the existence of measurable perfect matchings in hyperfinite bipartite graphings.
Recently, M\'ath\'e, Noel and Pikhurko \cite{mathe2022circle} managed to find a circle squaring with Jordan measurable pieces which are Boolean combinations of $F_{\sigma}$ sets. 

The Banach-Tarski paradox \cite{banach-tarski} shows that any two bounded subsets of $\mathbb{R}^3$ with nonempty interior are equidecomposable. 
Von Neumann \cite{vonneumann} introduced amenable groups in order to explain this paradox.
This is the largest class of groups where the existence of an equidecomposition may imply the existence of a measurable equidecomposition. 
Gardner made the following conjecture. 

\begin{conj} \cite{gardner}
Consider two bounded measurable sets $A, B \subseteq \mathbb{R}^n$ and an {amenable group} of isometries $\Gamma$. If $A$ and $B$ are $\Gamma$-equidecomposable then they admit a measurable equidecomposition.
\end{conj}

Gardner's conjecture would directly imply measurable circle squaring by the resulf of Laczkovich \cite{lacz11}.
Note that Gardner's conjecture does not require that the isometries in the measurable equidecomposition are in $\Gamma$. The second author has found an example of an amenable group of isometries $\Gamma$ and two measurable sets admitting a $\Gamma$-equidecomposition but no measurable $\Gamma$-equidecomposition; the example is a modification of Laczkovich's construction \cite{lacz3}. This example shows that in general the group of isometries admitting an equidecomposition may not admit a measurable equidecomposition. However, in this note we prove that this holds for sufficiently equidistributed sets (we call them $\Gamma$-uniform, motivated by \cite{lacz.uniform}); the only counterexamples are the groups admitting the infinite dihedral group (i.e., $\zz$) as a quotient by a finite subgroup.


Recall that given a Borel graph $G$ and a set $F\subseteq V(G),$
the \textit{boundary} of $F$ is $\partial F:= \{x\in F: \exists y\in V(G)\setminus F \textnormal{ with } (x,y) \in E(G) \}.$  A set $F$ is $\varepsilon$\textit{-F\o lner} if $|\partial F|< \varepsilon |F|$ for every finite $F \subset V(G)$.

\begin{definition}
Given a free pmp action of a finitely generated amenable group $\Gamma\curvearrowright
(X,\mu)$ and
$A\subseteq X$, we say that $A$ is \textit{$\Gamma$-uniform} if there exists 
$\varepsilon>0$ such that for every $\varepsilon$-F\o lner set $F$ contained in an orbit of the action, $|A\cap F|\geq \varepsilon|F|$ holds.
\end{definition}

Note that in the above definition, the property of being $\Gamma$-uniform does not depend on the choice of the generating set of the group $\Gamma$, but the constant $\varepsilon$ might.
Motivated by Gardner's conjecture we consider the following property.

\begin{definition}
An amenable group $\Gamma$ satisfies the \textit{uniform Gardner property} if whenever $\Gamma\curvearrowright
(X,\mu)$ is a free pmp action,
$A,B\subseteq X,$ and $A,B$ are \textit{$\Gamma$-uniform} and $\Gamma$-equidecomposable, then $A$ and $B$ admit a measurable $\Gamma$-equidecomposition.
\end{definition}

Now we state our main result. 

\begin{theorem}\label{weak gardner}
  Let $\Gamma$ be a finitely generated amenable group. $\Gamma$ satisfies the uniform Gardner property if and only if it does
  not admit $\zz$ as a quotient by a
    finite subgroup.
\end{theorem}




Marks and Unger \cite{marks-unger} 
used a bounded integral-valued Borel flow in the Schreier graph induced by finitely many translations
to find an equidecomposition of equidistributed sets. Our main technical contribution is a simple Borel algorithm for rounding flows that 
applies beyond the case of graphs arising from $\mathbb{Z}^d$ actions. It relies on the notion of connected toasts (see Definition \ref{toast}), which were introduced by the authors in \cite{bowen2021perfect}.  They are known to exist a.e. in all one-ended hyperfinite graphings, as well as in Borel actions of $\mathbb{Z}^d$. By a recent result of the first author, Poulin and Zomback \cite{bowen2022one} they also exist in Borel actions of polynomial growth groups. 

\begin{theorem}\label{borel}
Consider a locally finite Borel graph $G$, a bounded integral Borel function
$f: V(G) \mapsto \mathbb{Z}$ and a Borel $f$-flow $\phi$.
Assume that $G$ admits a Borel connected toast.
Then {$G$ has} an integral Borel $f$-flow $\psi$ such that $|\phi-\psi| < 3$ {holds everywhere}.
\end{theorem}

In particular, this implies that every locally finite one-ended hyperfinite graphing that admits a (not necessarily measurable) bounded real valued $f$-flow also admits a bounded integral measurable $f$-flow.  

\section{Definitions, notations}

{In this paper we will work with both \textit{Borel graphs}, i.e., graphs whose vertex set is a Polish space $X$ and whose edge set is a Borel subset of $X^2,$ and with \textit{graphings,} which are probability measure preserving Borel graphs.} We use the standard graph theoretical notation and refer to $X$ as $V(G)$ and denote the set of edges by $E(G)$. 

A locally finite connected graph is one-ended if after the removal of any finite subset of vertices it has exactly one infinite connected component. A locally finite infinite connected graph is two-ended if it is not one-ended and after the removal of any finite subset of vertices it has at most two infinite connected components. For every amenable group, its Cayley graph is either one-ended or two-ended, and the latter is true if and only if the group is virtually $\mathbb{Z}$. We write Ends($\Gamma$) for the set of ends of the group $\Gamma$.

We consider flows only on locally finite directed {Borel graphs} (to avoid imposing conditions such as absolute convergence of the sum of the flow at vertices of infinite degree). Given a locally finite directed Borel graph, a function $f:V(G)\rightarrow \mathbb{Z}$ and a function $c:E(G)\rightarrow \mathbb{N},$ we say that a function $\phi:E(G)\rightarrow \mathbb{R}$ is an \textit{$f$-flow bounded by $c$}  if $\phi(x,y)\leq c(x,y)$ for each  edge $(x,y)\in E(G)$, $\phi(x,y)=-\phi(y,x)$ and $\sum_{y\in N(x)}\phi(x,y)= f(x)$ for every  $x\in V(G)$.

For hyperfinite graphings the existence of a real valued measurable $f$-flow a.e. is equivalent to the existence of a not necessarily measurable $f$-flow, and this fact is implicit in the work of \cite{wehrung}, \cite{laczkovich.dec} and \cite{ciesla.sabok}.  

\begin{prop}\label{meas. flow}
Let $G$ be a locally finite hyperfinite graphing, $f:V(G)\rightarrow \mathbb{Z}$ a measurable function
and $\phi$ an $f$-flow bounded by the integral capacity $c$.  Then $G$ admits a measurable $f$-flow $\psi$ bounded by $c$. \end{prop}

See \cite{bowen2021perfect} Lemma 2.4 for a proof of the nearly identical fact with fractional matchings taking the place of flows. We include the proof for flows for the sake of completeness.

\begin{proof}[Proof of Lemma \ref{meas. flow}]
Using the hyperfiniteness of $G$, take a sequence of measurable partitions $\mathcal{C}_n$ of measurable subsets $X_n\subseteq X$ into finite subsets such that $\nu(X_n)\geq 1-\frac{1}{2^n}$. 
We use this to construct a sequence of measurable functions that are  $f$-flows on elements of $\mathcal{C}_n$. For each element $F$ of $\mathcal{C}_n$ choose the 
lexicographically least integral function $\varphi_n(F)$ on the edges in $F$ which can be extended to a $f$-flow bounded by $c$.
For each $n$, take the union of all $\varphi_n(F)$ for $F\in \mathcal{C}_n$. This gives a function which is a measurable $f$-flow on the vertices that are in the interior of a cell in $\mathcal{C}_n$ for all but finitely many $n$. This sequence has a weakly convergent subsequence in $L^2(E(G))$, and the limit satisfies the conditions of the lemma.
\end{proof}

We will also use the following basic lemma about finite graphs. 

\begin{lemma}\label{subgraph}
Let $G$ be a finite, connected graph and $P \subseteq V(G)$ a subset of even size.
Then there exists a spanning subgraph $H$ of $G$ such that every vertex of $P$ has odd degree in $H$, and every vertex of  $V(G) \setminus P$ has even degree in $H$.
\end{lemma}

\begin{proof}
We prove by induction on the size of $P$. If $P=\emptyset$ then $H$ can be the edgeless spanning subgraph. Else choose two different vertices $s,t \in P$. 
Let $L$ denote a path connecting $s$ and $t$. By induction, there exists a spanning subgraph $H'$ of $G$ such that the degree of the vertices in $P \setminus \{ s,t \}$ is odd, and the degree of the other vertices is even in $H'$. Let $E(H)$ be the symmetric difference of $E(L)$ and $E(H')$.
\end{proof}

\subsection{{Connected toasts}}

\hspace{1mm}

\vspace{2mm}

The following definition \cite{bowen2021perfect} refines the notion of a toast (see \cite[Definition 2.9]{gjks}, \cite[Definition 4.1]{gjks.forcing}), coined by Miller and motivated by the work of Conley and Miller \cite{conley.miller.toast}. In particular, a \textit{toast} is any tiling that satisfies Properties (1) and (2) of the definition below.

\begin{definition}\label{toast}
Given a Borel graph $G$, we say that a Borel collection $\mathcal{T}$ of finite connected subsets of $V(G)$ is a \textit{connected toast} if it 
  satisfies

\begin{enumerate}
\item{$\bigcup_{K\in \mathcal{T}}E(K)=E(G)$,} 

\item{for every pair $K, L \in \mathcal{T}$ either $(K\cup N(K)) \cap L = \emptyset$ or $K \cup N(K) \subseteq L$, or }$L \cup N(L) \subseteq K$,

\item{for every $K \in \mathcal{T}$ the induced subgraph on $K \setminus \bigcup_{K \supsetneq L \in \mathcal{T}} L$ is connected.}
\end{enumerate}

\end{definition}


{The proof of Theorem \ref{weak gardner} will use the following result on connected toasts.}

\begin{prop}\label{tiling}\cite{bowen2021perfect}
Any locally finite, one-ended hyperfinite graphing admits a connected toast a.e. 
\end{prop}

In order to find Borel flows and equidecompositions we will need connected toasts everywhere. The following lemma provides this for actions of $\mathbb{Z}^d$, which is enough for most geometric applications including circle squaring. Recently, the first author, Poulin and Zomback \cite{bowen2022one}  generalized this to every one-ended polynomial growth group using the ideas of this proof. We include the proof for $\mathbb{Z}^d$.

\begin{prop}\label{one-ended.zd}
{The Schreier graph of} any free Borel action of $\mathbb{Z}^d$, for $d\geq2$, admits a connected toast.
\end{prop} 


\begin{proof}
Recall that a \textit{$k$-toast} in a Borel graph $G$ is a Borel collection $\mathcal{T}$ of finite connected subsets of $V(G)$ such that $\bigcup_{K\in \mathcal{T}}E(K)=E(G)$, and for every pair $K, L \in \mathcal{T}$ either $(N^k(K)\cup K) \cap L = \emptyset$ or $K \cup N^k(K) \subseteq L$, or $L \cup N^k(L) \subseteq K$. By the result of Gao, Jackson, Krohne, and Seward (see \cite[Theorem 5.5]{marks-unger}), any free Borel $\mathbb{Z}^d$ action admits a $k$-toast $\mathcal{T}$ such that the boundary of each $K\in\mathcal{T}$ is visible from infinity (i.e., $K$ has no holes). Note that $\mathcal{T}$ satisfies Properties (1) and (2) of Definition \ref{toast}, 
and we claim that it also satisfies Property (3) as long as $k>2.$  {To see this, we will use a result of Deuschel and Pisztora \cite[Lemma 2.1(i)]{visual.boundary} which says that if a {finite, connected} set $H\subset \mathbb{Z}^d$ has no holes, then $N^2(H)\setminus H$ is connected. {Thus,} 
since each $K\in \mathcal{T}$ has no holes, the set $N^2(K)\setminus K$ is connected, and so for every $L \in \mathcal{T}$ the induced subgraph on $L \setminus \bigcup_{L \supsetneq K \in \mathcal{T}} K$ is connected, as any path through smaller tiles $K$ can be rerouted through $N^2(K)\setminus K$.}
\end{proof}


\section{Integral Borel flows using Borel tilings}
\label{sec:borel-flows}

In this  section we give a short algorithm that allows us to convert bounded real-valued Borel flows into bounded integral Borel flows, assuming that the graph admits a connected toast. This is applicable to finding Borel equidecompositions. Namely, as shown by Marks and Unger in \cite{marks-unger}, the existence of bounded integral Borel flows can be used to prove equidecomposition results using tilings such as the Gao--Jackson tilings \cite{gao.jackson} for actions of $\mathbb{Z}^d$. In the next section we will use these ideas to prove our results on measurable equidecompositions.



We state the theorem below in the setting of Borel graphs that admit a connected toast and Borel real-valued flow, as it applies to graphs for which such connected toasts can be found. In particular, the main obstacle towards using Theorem \ref{borel} to produce Borel  equidecompositions lies with finding Borel real-valued flows. 


\begin{theorem}\label{borel}
Consider a locally finite Borel graph $G$, a bounded integral Borel function
$f: V(G) \mapsto \mathbb{Z}$ {\color{blue}} and a Borel $f$-flow $\phi$.
Assume that $G$ admits a Borel connected toast.
Then there is an integral Borel $f$-flow $\psi$ such that $|\phi-\psi| < 3$.
\end{theorem}

\begin{proof}

The theorem follows from the two lemmas below.

\begin{lemma} 
Consider a locally finite Borel graph $G$, a bounded integral Borel function
$f: V(G) \mapsto \mathbb{Z}$ and a Borel $f$-flow $\phi$.
Assume that $G$ admits a Borel connected toast.
Then there is a dyadic-valued Borel $f$-flow $\psi$ such that $|\phi-\psi| < 1$.
\end{lemma}

\begin{proof}
Let $\mathcal{T}$ be a Borel connected toast and let $\mathcal{M}_1 \subset \mathcal{T}$ denote the family of minimal sets (ordered by containment),
$\mathcal{M}_2$ denote the family of minimal sets in $\mathcal{T} \setminus \mathcal{M}_1$ etc.
Obviously $\mathcal{T} = \cup_{k=1}^{\infty} \mathcal{M}_k$. We define a sequence of flows recursively.
Set $\psi_1=\phi$. Assume that $\psi_k$ is a Borel $f$-flow, $\psi_k$ is dyadic-valued on 
$\bigcup_{K \in \mathcal{M}_{k-1}} E(K)$ and $|\psi_k-\phi|<1$. Choose a set $K \in \mathcal{M}_{k+1}$.
While there exists an $L \subset K, L \in \mathcal{M}_k$ where the value of the flow is not dyadic choose a 
cycle $C$ in $L \cup (K \setminus \bigcup\mathcal{M}_k)$ such that the value of the flow on the
edges of $C$ in $L$ is not dyadic. We can find such a cycle, since $K \setminus \bigcup\mathcal{M}_k$
is connected and there is no vertex of degree one in $\bigcup\mathcal{M}_k$  in the spanning subgraph of non-dyadic edges.
Add a circuit to the flow that is constant on this (oriented) cycle and zero elsewhere in such a way that after adding this circuit at least one edge on the cycle becomes dyadic and if we are at step $n$ of the whole construction, then the constant is smaller than $\frac{1}{2^n}$.
After each step, the number of edges in $L$ where the value of the flow is dyadic increases. 
Hence the process will finish in at most $|E(K)|$ steps for every $K$. The resulting $f$-flow $\psi_{k+1}$ can be Borel and it is close enough: $|\phi-\psi_{k+1}|<1$ as the added circuits were small enough. Note that the value of the flow does not change once it becomes dyadic, 
hence the sequence $\psi_k$ will stabilize for every edge. The limit $\psi$ satisfies the conditions of the lemma.
\end{proof}

\begin{lemma}
Consider a locally finite Borel graph $G$, an integral bounded Borel function
$f: V(G) \mapsto \mathbb{Z}$ and a dyadic-valued Borel $f$-flow $\phi$.
Assume that $G$ admits a Borel connected toast.
Then there is an integral Borel $f$-flow $\psi$ such that $|\phi-\psi| < 2$.
\end{lemma}

\begin{proof}
The proof is similar to the previous one. Let $\mathcal{T}$ be a Borel connected toast and let $\mathcal{M}_1 \subseteq \mathcal{T}$ denote the family of minimal sets,
$\mathcal{M}_2$ denote the family of minimal sets in $\mathcal{T} \setminus \mathcal{M}_1$ etc. We define a sequence of flows recursively,
starting with $\psi_1=\phi$. Assume that $\psi_k$ is a Borel $f$-flow and $\psi_k$ is integral on $\bigcup_{K \in \mathcal{M}_{k-1}} E(K)$.
Consider a tile $K \in \mathcal{M}_{k+1}$. Let $2^l$ be the largest denominator of the flow value over every edge in 
$\bigcup_{L\subseteq K, L \in \mathcal{M}_k} E(L)$. If $l > 0$ then consider the set $P$ of vertices in $K \setminus \bigcup \mathcal{M}_k$
that have an odd number of neighbors in $\bigcup_{L\subseteq K, L \in \mathcal{M}_k} L$ such that the value of the flow on the corresponding edge 
has denominator $2^l$. Using Lemma~\ref{subgraph} applied to $P$ and $K \setminus \bigcup\mathcal{M}_k$, we get a subgraph $H_K$ on $K$ such that for every vertex in $K \setminus \bigcup\mathcal{M}_k$ its degree is even if and only if the vertex does not belong to $P$. Taking the union of $H_K$ with the graph
spanned by the vertices with at least one endvertex in $\bigcup_{L\subseteq K, L \in \mathcal{M}_k} V(L)$, where the denominator of the flow is equal to $2^l$, we get a graph whose vertices have all even degree. 
This graph partitions into a family of edge-disjoint cycles that covers every edge in $\bigcup_{L \subset K, L \in \mathcal{M}_k} E(L)$ with the flow value of denominator $2^l$ exactly once, and no other edge in 
$\bigcup_{L \subset K, L \in \mathcal{M}_k} E(L)$. Choose an orientation of every cycle and add $2^{-l}$ to the value of the flow on these cycles. 
Now the largest denominator of the flow values in $\bigcup_{L \subset K, L \in \mathcal{M}_k} E(L)$ is at most $2^{l-1}$.
Continue this until every edge in $\bigcup_{L \in \mathcal{M}_k} E(L)$ becomes integral.
The resulting $f$-flow $\psi_{k+1}$ can be constructed to be Borel, since the cycles and their orientations can be chosen in a Borel way. The value of the flow does 
not change once it becomes integral, hence the sequence $\psi_k$ will stabilize for every edge. Consider an edge $e \in E(K)$ and the smallest $k$
for which there exists $K \in \mathcal{M}_k$ such that $e \in E(K)$. Note that $\psi_{k+1}(e)=\psi_{k+2}(e)=\dots$, and $\phi(e)=\psi_{k-1}(e)$.
$|\psi_{k-1}(e)-\psi_k(e)|<1$, since the value of the flow at $e$ can change by less than $\sum_{l=1}^{\infty} 2^{-l}$. Similarly, $|\psi_k(e)-\psi_{k+1}(e)|<1$. Hence the limit $\psi=\lim_{n \rightarrow \infty} \psi_n$ satisfies the conditions of the lemma.
\end{proof}

\end{proof}
If we are interested in measurable rather than Borel flows (as we will be in the next section), we can use the above Theorem to deduce the following more general result on obtaining measurable integral flows from (possibly) non-measurable ones. 

\begin{cor}\label{measurable flows}
Let $G$ be a locally finite one-ended hyperfinite graphing and $f:V(G)\rightarrow \mathbb{Z}$ an integral bounded measurable function. If $G$ admits a (not necessarily measurable) real-valued bounded $f$-flow then $G$ admits an integral bounded measurable $f$-flow.
\end{cor}

\begin{proof}
Such graphings admit connected toasts a.e. by Proposition \ref{tiling}. Moreover, by Proposition \ref{meas. flow} the existence of a (not necessarily measurable) bounded real valued $f$-flow implies the existence of a measurable bounded real valued $f$-flow.  These facts together with Theorem \ref{borel} give the desired result.
\end{proof}

In particular, the above result applies to the Schreier graph of any free pmp action of a finitely generated amenable group that is not virtually $\mathbb{Z}.$

\section{The uniform Gardner conjecture and the infinite dihedral group}
\label{sec:equidecompositions}

In this section we prove Theorem \ref{weak gardner}.  As mentioned in the introduction, the positive part of our result will follow from Theorem \ref{measurable flows} and ideas from Marks and Unger's Borel circle squaring proof \cite{marks-unger}.  In the negative direction, the novelty is the identification of $\zz$ as essentially the only group that does not satisfy this strengthening of the Gardner conjecture. 


 
We now show that the only groups without the uniform Gardner property are those groups which admit $\zz$ as a finite quotient. Recall that an action of an infinite group is {\it totally ergodic} if every element of infinite order acts ergodically.

\begin{lemma}\label{totally.ergodic}
Let $\zz\curvearrowright X$ be a free pmp totally ergodic action. There is
no measurable choice of an end $e:X\to\e(\zz)$ which is invariant
under the action.
\end{lemma}

\begin{proof}
Suppose for a contradiction that $e:X\to\e(\zz)$ is a measurable invariant
choice of end. Write $\alpha,\beta$ for the generators of
$\zz=\langle \alpha,\beta|\alpha^2=\beta^2=1\rangle$. Then for every $x\in
X$, the function $e$ induces a choice of one of the
generators $\alpha,\beta$ pointing from $x$ in the direction of
$e(x)$. Write $A\subseteq X$ for the set of points which
choose $\alpha$ and $B\subseteq X$ for the set of points which
choose $\beta$. Note that $A$ and $B$ are measurable since $e$
was measurable. Also, they alternate on every orbit,
namely $\alpha A=B$ and $\beta B=A$. Then both $A$ and $B$ are
invariant under $\alpha \beta$. However, since every element of
infinite order acts ergodically, the sets $A$
and $B$ should have measure one or zero, which is
impossible, since $\alpha$ and $\beta$ are measure-preserving.
\end{proof}

\begin{lemma}\label{zz gard}
   Suppose that $\Gamma$ admits $\zz$ as a quotient by a finite normal subgroup and $\Gamma\curvearrowright X$ is a free pmp totally ergodic action. Then there are two measurable $\Gamma$-uniform sets of the same measure
  in $X$ which are $\Gamma$-equidecomposable but not measurably $\Gamma$-equidecomposable.
\end{lemma}

\begin{proof}
Let $\Delta\lhd\Gamma$ be finite such that $\Gamma\slash\Delta=\zz$. Write $X'=X\slash\Delta$. Let $G$ be the Schreier graph of the action of $\Gamma$ on $X$ and let $G'$ be the Schreier graph of the action of $\zz=\langle \alpha,\beta|\alpha^2=\beta^2=1\rangle$ on $X'$. 
Let $A'\subseteq X'$ be a maximal Borel set which is an independent set in $G'$. Note that on each orbit of $\zz$, the set $A'$ has gaps of size either $1$ or $2$. Find a  Borel set $B'$ that picks exactly one point in each of the gaps left by $A'$.

Let $A\subseteq X$ be a Borel set meeting each $\Delta a'$ in exactly one point for every $a'\in A'$, and let $B\subseteq X$ be a Borel set meeting each $\Delta b'$ in exactly one point for every $b'\in B$. Note that the sets $A$ and $B$ have the same positive measure and are $\Gamma$-uniform.

To see that $A$ and $B$ are not measurably $\Gamma$-equidecomposable, we first argue about $A'$ and $B'$.

\begin{claim}
$A'$ and $B'$ are not measurably $\zz$-equidecomposable.
\end{claim}

\begin{proof}
Indeed, suppose that such a measurable $\zz$-equide\-composition exists. Then using such an equidecomposition, we can form a measurable bounded integer $G'$-flow $\psi$ for the function $\chi_{A'}-\chi_{B'}$. Such a flow induces a measurable perfect matching between $A'$ and $B'$ by matching $x\in A'$ to the point in $B'$ that is on the side of the edge $(x,y)$ adjacent to $x$ for which $\psi(x,y)>0$. Now, a measurable perfect matching between $A'$ and $B'$ leads to a measurable invariant choice of an end for this action. Namely, the matching restricted to every
orbit separately has to match every point $x\in A'$ to the point $y\in B'$ such that the direction from $x$ to $y$ determines one of the ends. This choice of an end does not depend on $x\in A'$, so it is invariant under the action. As we assume that the action is totally ergodic, we get a contradiction with Lemma \ref{totally.ergodic}.
\end{proof}

Now, the above claim gives that there is no $\Gamma$-equidecomposition between $A$ and $B$ because such an equidecomposition would lead to a measurable
$\zz$-equidecomposition between $A'$ and $B'$.
\end{proof}

\begin{cor}\label{zz gardner cor}
  If $\Gamma$ admits $\zz$ as a quotient by a finite normal
  subgroup then $\Gamma$ does not have the uniform Gardner property.
\end{cor}

\begin{proof}
This follows directly from  Lemma \ref{zz gard} by considering a totally ergodic action of $\Gamma$.
\end{proof}

Now we deal with groups which have the uniform Gardner property.

\begin{prop}\label{one-ended.wg}
  Let $\Gamma$ be an amenable group. If $\Gamma$ is one-ended or has a finite normal subgroup $\Delta\lhd \Gamma$ with $\Gamma\slash\Delta=\mathbb{Z}$, then it
  has the uniform Gardner property.
\end{prop}


\begin{proof}
Suppose $\Gamma\curvearrowright X$ is a free Borel action and $A,B\subseteq X$ are measurable,
$\Gamma$-uniform and equidecomposable. Let $c$ be a constant witnessing that $A$ and $B$ are $\Gamma$-uniform. 
  
Note that if $\Gamma\slash\Delta=\mathbb{Z}$ for a finite normal subgroup $\Delta\lhd\Gamma$, then any free action $\Gamma\curvearrowright X$ admits a measurable choice of an end in any orbit. Since $A$ and $B$ are equidecomposable, there exists an integral $(\chi_A-\chi_B)$-flow bounded by $1$. We can also get a measurable, integral $(\chi_A-\chi_B)$-flow $\psi$ bounded by $4$ in the one-ended case by Corollary \ref{measurable flows}.
  
By the result of Conley, Jackson, Kerr, Marks, Seward and Tucker-Drob \cite{folner.tilings}, we can choose
a tiling $\mathcal{T}$ of the space with
$\frac{c}{4}$-F\o lner sets. Since $A$ and $B$ are $\Gamma$-uniform, we have
$|T\cap A|,|T\cap B|\geq c|T|$ for every tile $T\in\mathcal{T}$. 
  
For every pair of tiles $T,S\in\mathcal{T}$ which are adjacent put $\psi(S,T)$ to be the sum of all $\psi(e)$ for edges $e$ starting at a point in $S$ and ending at a point in $T$. Note that for each tile $T\in\mathcal{T}$ we have that $|A\cap T|,|B\cap T|$ are bigger than the sum of $\psi(T,S)$ for all tiles $S$ adjacent to $T$. Therefore, we can move points in a measurable way from $A$ and $B$ between the tiles as indicated by $\psi$, so after that in each tile there will be the same number of points from $A$ and $B$. Any measurable choice of bijections in the tiles gives a measurable equidecomposition.
\end{proof}
Finally, we prove Theorem \ref{weak gardner}.

\begin{proof}[Proof of Theorem \ref{weak gardner}]
  Any infinite amenable group has one or two ends (see for instance \cite[Th\'{e}or\`{e}me 10 and Chapitre V]{poenaru}). First
  note that in the one-ended case $\Gamma$ does not admit
  $\zz$ as a quotient by a finite subgroup and it satisfies
  the uniform Gardner property by Proposition \ref{one-ended.wg}.

  Thus, we only focus on the two-ended case. If $\Gamma$ has
  two ends then by \cite[Theorem 5.12]{scott.wall} it has a finite normal subgroup
  $\Delta$ such that $\Gamma\slash\Delta$ is either
  isomorphic to $\mathbb{Z}$ or to $\zz$. If $\Gamma\slash\Delta=\mathbb{Z}$, then we are done again by Proposition \ref{one-ended.wg}. In the second case, Corollary \ref{zz gardner cor} implies the theorem.
\end{proof}

\bibliographystyle{amsalpha} 
\bibliography{1}

\newcommand{\etalchar}[1]{$^{#1}$}
\providecommand{\bysame}{\leavevmode\hbox to3em{\hrulefill}\thinspace}
\providecommand{\MR}{\relax\ifhmode\unskip\space\fi MR }
\providecommand{\MRhref}[2]{%
  \href{http://www.ams.org/mathscinet-getitem?mr=#1}{#2}
}
\providecommand{\href}[2]{#2}
\begin{thebibliography}{CJK{\etalchar{+}}18}

\bibitem[BKS21]{bowen2021perfect}
Matthew Bowen, Gabor Kun, and Marcin Sabok, \emph{Perfect matchings in
  hyperfinite graphings}, arXiv preprint arXiv:2106.01988 (2021).

\bibitem[BPZ22]{bowen2022one}
Matthew Bowen, Antoine Poulin, and Jenna Zomback, \emph{One-ended spanning
  trees and definable combinatorics}, arXiv preprint arXiv:2210.14300 (2022).

\bibitem[BT24]{banach-tarski}
S.~Banach and A.~Tarski, \emph{Sur la d\'ecomposition des ensembles de points
  en parties respectivement congruentes}, Fund. Math \textbf{6} (1924), no.~1,
  244--277.

\bibitem[CJK{\etalchar{+}}18]{folner.tilings}
Clinton~T. Conley, Steve~C. Jackson, David Kerr, Andrew~S. Marks, Brandon
  Seward, and Robin~D. Tucker-Drob, \emph{F\o lner tilings for actions of
  amenable groups}, Math. Ann. \textbf{371} (2018), no.~1-2, 663--683.

\bibitem[CM16]{conley.miller.toast}
Clinton~T. Conley and Benjamin~D. Miller, \emph{A bound on measurable chromatic
  numbers of locally finite {B}orel graphs}, Math. Res. Lett. \textbf{23}
  (2016), no.~6, 1633--1644.

\bibitem[CS]{ciesla.sabok}
Tomasz Cie\'{s}la and Marcin Sabok, \emph{Measurable {H}all's theorem for
  actions of abelian groups}, preprint, arXiv:1903.02987.

\bibitem[DP96]{visual.boundary}
Jean-Dominique Deuschel and Agoston Pisztora, \emph{Surface order large
  deviations for high-density percolation}, Probab. Theory Related Fields
  \textbf{104} (1996), no.~4, 467--482.

\bibitem[Gar91]{gardner}
R.~J. Gardner, \emph{Measure theory and some problems in geometry}, Atti Sem.
  Mat. Fis. Univ. Modena \textbf{39} (1991), no.~1, 51--72.

\bibitem[GJ15]{gao.jackson}
Su~Gao and Steve Jackson, \emph{Countable abelian group actions and hyperfinite
  equivalence relations}, Invent. Math. \textbf{201} (2015), no.~1, 309--383.

\bibitem[GJKS15]{gjks.forcing}
S.~Gao, S.~Jackson, E.~Krohne, and B~Seward, \emph{Forcing constructions and
  countable borel equivalence relations}, preprint, 2015, available at
  \texttt{https://itservices.cas.unt.edu/~sgao/pub/pub.html}.

\bibitem[GJKS18]{gjks}
\bysame, \emph{Continuous combinatorics of abelian group actions}, preprint,
  arXiv:1803.03872.

\bibitem[GMP17]{gmp}
\L{}ukasz Grabowski, Andr\'{a}s M\'{a}th\'{e}, and Oleg Pikhurko,
  \emph{Measurable circle squaring}, Ann. of Math. (2) \textbf{185} (2017),
  no.~2, 671--710.

\bibitem[Lac88]{lacz3}
Mikl\'{o}s Laczkovich, \emph{Closed sets without measurable matching}, Proc.
  Amer. Math. Soc. \textbf{103} (1988), no.~3, 894--896.

\bibitem[Lac90]{lacz11}
\bysame, \emph{Equidecomposability and discrepancy; a solution of {T}arski's
  circle-squaring problem}, J. Reine Angew. Math. \textbf{404} (1990), 77--117.

\bibitem[Lac92]{lacz.uniform}
\bysame, \emph{Uniformly spread discrete sets in {${\bf R}^d$}}, J. London
  Math. Soc. (2) \textbf{46} (1992), no.~1, 39--57.

\bibitem[Lac96]{laczkovich.dec}
\bysame, \emph{Decomposition using measurable functions}, C. R. Acad. Sci.
  Paris S\'{e}r. I Math. \textbf{323} (1996), no.~6, 583--586.

\bibitem[M{\'{a}}t18]{MatheICM}
Andr\'{a}s M{\'{a}}th\'{e}, \emph{Measurable equidecompositions}, Proceedings
  of the {I}nternational {C}ongress of {M}athematicians---{R}io de {J}aneiro
  2018. {V}ol. {III}. {I}nvited lectures, World Sci. Publ., Hackensack, NJ,
  2018, pp.~1713--1731.

\bibitem[MNP22]{mathe2022circle}
Andr{\'a}s M{\'a}th{\'e}, Jonathan~A Noel, and Oleg Pikhurko, \emph{Circle
  squaring with pieces of small boundary and low borel complexity}, arXiv
  preprint arXiv:2202.01412 (2022).

\bibitem[MU17]{marks-unger}
Andrew~S. Marks and Spencer~T. Unger, \emph{Borel circle squaring}, Ann. of
  Math. (2) \textbf{186} (2017), no.~2, 581--605.

\bibitem[Po{\'{e}}74]{poenaru}
Valentin Po{\'{e}}naru, \emph{Groupes discrets}, Lecture Notes in Mathematics,
  Vol. 421, Springer-Verlag, Berlin-New York, 1974.

\bibitem[SW79]{scott.wall}
Peter Scott and Terry Wall, \emph{Topological methods in group theory},
  Homological group theory ({P}roc. {S}ympos., {D}urham, 1977), London Math.
  Soc. Lecture Note Ser., vol.~36, Cambridge Univ. Press, Cambridge-New York,
  1979, pp.~137--203.

\bibitem[Tar25]{Tarski.problem}
Alfred Tarski, \emph{Probl\'eme 38}, Fund. Math. \textbf{7} (1925).

\bibitem[vN29]{vonneumann}
John von Neumann, \emph{Zur allgemeinen theorie des masses}, Fund. Math.
  \textbf{13} (1929), no.~1, 73--116.

\bibitem[Weh92]{wehrung}
Friedrich Wehrung, \emph{Injective positively ordered monoids. {I}, {II}}, J.
  Pure Appl. Algebra \textbf{83} (1992), no.~1, 43--82, 83--100.

\end{thebibliography}

\end{document}